\documentclass[11pt]{amsart}

\usepackage{geometry}
\usepackage{amsmath,amssymb,amsthm,mathtools}
\usepackage{enumitem}
\usepackage[colorlinks=true,linkcolor=blue,citecolor=blue,urlcolor=blue]{hyperref}

\newtheorem{theorem}{Theorem}
\newtheorem{lemma}[theorem]{Lemma}
\newtheorem{proposition}[theorem]{Proposition}

\newtheorem{remark}[theorem]{Remark}
\newtheorem{question}[theorem]{Question}

\DeclareMathOperator{\Pic}{Pic}
\DeclareMathOperator{\NS}{NS}

\newcommand{\DD}{\mathcal D}
\newcommand{\Gm}{\mathbb G_m}

\title{A note on Azumaya algebras and one-forms}
\author{Siqing Zhang}
\date{}

\begin{document}

\begin{abstract}
The crystalline differential operators on a smooth variety $X$ give rise to a non-split Azumaya algebra over the cotangent bundle of the Frobenius twist $X'$.
In some cases, this Azumaya algebra splits when restricted to finite covers of $X'$.
In this short note, we show that, whenever $X$ has a non-closed global one-form, there is a degree one cover of $X'$ on which the Azumaya algebra does not split, answering a question of Sasha Petrov.
\end{abstract}

\maketitle

\section{The Question}

Throughout \(k\) is an algebraically closed field of characteristic \(p>0\).
Let \(X\) be a smooth proper variety over \(k\), and let \(X'\) be its
Frobenius twist.  By Bezrukavnikov--Mirkovic--Rumynin \cite[Theorem
2.2.3]{BMR}, the sheaf of crystalline differential operators on \(X\) defines
an Azumaya algebra \(\mathcal A_X\) on the cotangent bundle  \(T^*X'\).

In many useful cases, especially in the characteristic-\(p\) geometric
Langlands and non-abelian Hodge Theory for curves
\cite{BB,Groechenig,deCataldoZhang}, and for abelian varieties
\cite{Hao}, this Azumaya algebra splits when
restricted to spectral varieties; compare also the \(p\)-adic analogue
\cite{Heuer}.
Sasha Petrov asked the following question:

\begin{question}\cite[Question 7.3(1)]{Petrov}\label{petrov-question}
Does there exist a smooth proper variety \(X\) together with a smooth
subvariety \(Z\subset T^*X'\), finite over \(X'\), such that
\(\mathcal A_X|_Z\) does not split?
\end{question}

In this note, we answer this question affirmatively in every positive characteristic.

\section{The Answer}

 Let
\(
X' := X\times_{k,F_k} k
\)
be the Frobenius twist.
Write \(F_X=F_{X/k}\colon X\to X'\) for the relative
Frobenius, and \(\pi_X\colon X'\to X\) for the projection.
Milne's exact sequence, in the form used by Ogus--Vologodsky
\cite[(4.1.1)]{OV}, is the following exact sequence of \'{e}tale sheaves of
abelian groups on \(X'\):
\begin{equation}\label{eq:milne}
0\longrightarrow \mathcal O_{X'}^*
\longrightarrow F_{X,*}\mathcal O_X^*
\xrightarrow{d\log}
F_{X,*}Z^1_{X/k}
\xrightarrow{\pi_X^*-C_X}
\Omega^1_{X'/k}
\longrightarrow 0.
\end{equation}
Here \(Z^1_{X/k}\subset \Omega^1_{X/k}\) is the sheaf of closed one-forms and
\(C_X\) is the Cartier operator.  Splice the exact sequence and the connecting morphisms define the obstruction homomorphism
\begin{equation}\label{eq:obstruction}
\operatorname{ob}_{X'}\colon
H^0(X',\Omega^1_{X'/k})
\to
H^2_{\mathrm{\acute et}}(X',\mathcal O_{X'}^*).
\end{equation}

Let $\DD_{X/k}$ be the sheaf of crystalline differential operators on $X$.
 The center
of \(F_{X,*}\DD_{X/k}\) is identified with
\(\operatorname{Sym}_{\mathcal O_{X'}}T_{X'/k}\), and 
\(F_{X,*}\DD_{X/k}\) is identified with an Azumaya algebra on \(T^*X'\).  This is the algebra
denoted \(\mathcal A_X\) above.

\begin{lemma}\label{lem:graph}
Let \(X/k\) be smooth, and let
\(\omega\in H^0(X',\Omega^1_{X'/k})\).  Let
\(
i_\omega\colon X'\longrightarrow T^*X'
\)
be the section corresponding to \(\omega\), with graph
\(\Gamma_\omega\subset T^*X'\).  Then
\[
[i_\omega^*\mathcal A_X]
=\operatorname{ob}_{X'}(\omega)
\quad\text{in } H^2_{\mathrm{\acute et}}(X',\mathcal O_{X'}^*).
\]
In particular, if \(\operatorname{ob}_{X'}(\omega)\ne 0\), then
\(\mathcal A_X|_{\Gamma_\omega}\) does not split.
\end{lemma}

\begin{proof}
Ogus--Vologodsky identify \(\operatorname{ob}_{X'}(\omega)\) with the class of
the \(\Gm\)-gerbe of line bundles on \(X\) with integrable connection and
\(p\)-curvature \(\omega\) \cite[Proposition 4.2]{OV}.  The same gerbe is the gerbe of splittings of the restriction of
\(\mathcal A_X\) to the section \(i_\omega\) by \cite[Remark 4.3]{OV}.  
\end{proof}

The following proposition reduces Question \ref{petrov-question} to finding a variety with non-closed global 1-forms.

\begin{proposition}\label{prop:dimension}
Let \(X/k\) be a smooth proper connected variety.  If not every global
one-form on \(X\) is closed, i.e.
\(
H^0(X,Z^1_{X/k})\subsetneq H^0(X,\Omega^1_{X/k}),
\)
then the obstruction map
\(
\operatorname{ob}_{X'}\) as in \eqref{eq:obstruction}
is nonzero.
\end{proposition}

\begin{proof}
Let \(G:=\Pic^{\nabla,\mathrm{rig}}_{X/k}\)
be the rigidified Picard scheme of line bundles with integrable
connection.  Let \(V:=H^0(X',\Omega^1_{X'/k})\), viewed as a vector group.
By \cite[Proposition 4.11 and the paragraph following it]{OV}, taking \(p\)-curvature defines a morphism of group schemes
\[
\psi\colon G\longrightarrow V.
\]
By \cite[Proposition 4.2]{OV}, \(\operatorname{ob}_{X'}(\omega)=0\) if and only if \(\omega\in\psi(G(k))\).  Thus it is enough to show that
\(\psi(G(k))\ne V(k)\).

By Cartier descent,
\[
\Pic_{X'/k}\longrightarrow G,\qquad
M\longmapsto (F_X^*M,\nabla_{\mathrm{can}})
\]
identifies \(\Pic_{X'/k}\) with \(\ker(\psi)\); see \cite[Theorem 5.1]{Katz}.
Forgetting the connection fits into the exact sequence of fppf sheaves
\begin{equation*}
0\to H^0(X,Z^1_{X/k})\to G\xrightarrow{b}\Pic_{X/k}
\to H^1(X,Z^1_{X/k}),
\end{equation*}
as in \cite[Proposition 4.11]{OV}.  In particular, if \(G^0\) is the neutral
component of \(G\), then
\[
\dim G^0\leq h^0(X,Z^1_{X/k})+\dim \Pic^0_{X/k}.
\]
Let \(\operatorname{Im}(\psi|_{G^0})\) be the algebraic subgroup image.
The quotient morphism \(G^0\to \operatorname{Im}(\psi|_{G^0})\) is 
faithfully flat. Therefore
\[
\dim \operatorname{Im}(\psi|_{G^0})=\dim G^0-\dim\ker(\psi|_{G^0})
\leq h^0(X,Z^1_{X/k})<\dim V.
\]

Finally, passing from \(G^0\) to \(G\) only gives finitely many translates:
Let \(G_1=b^{-1}(\Pic^0_{X/k})\).  Since \(G_1\) is of finite type,
\(G_1/G^0\) is finite.  Moreover \(G(k)/G_1(k)\) injects into \(\NS(X)\), so
its image under \(\psi\) is finite: \(\NS(X)\) is finitely generated, whereas
\(V(k)\) is killed by \(p\).

\end{proof}

\begin{theorem}
    Let $X/k$ be a smooth proper connected variety.
    If not every global one-form on $X$ is closed, then there is a smooth subvariety $Z\subset T^*X'$, finite over $X'$, such that $\mathcal{A}_X|_Z$ does not split. Moreover, the finite morphism $Z\to X'$ is an isomorphism.
\end{theorem}

\begin{proof}
    Combine Lemma \ref{lem:graph} and Proposition \ref{prop:dimension}, and take $Z$ to be $\Gamma_{\omega}$ for some $\omega$ such that $\operatorname{ob}_{X'}(\omega)\neq 0$.
\end{proof}

\subsection{The Examples}\;

\subsubsection{Surface examples}\;

The existence of smooth projective surfaces with non-closed global
one-forms goes back at least to Mumford \cite[Corollary on p.341]{MumfordPathologies}.
The author would like to thank Sasha Petrov for pointing out this example.

In \cite[Corollary 3.4 and the paragraph following it]{Takeda}, Takeda constructs some generalized Raynaud surfaces that also have non-closed global one-forms.
\cite{Takeda} is written under the assumption that $p>2$, but one can check that these examples still hold when $p=2$, see e.g. \cite[\S5]{Tziolas} for more details.

\subsubsection{Higher-dimensional examples}\;

For higher-dimensional examples, we can take $X$ to be any of the surfaces above and take 
$X\times\mathbf P^n_k$.
Indeed, if $\eta\in H^0(X,\Omega_{X/k}^1)$ is non-closed, then so is the pullback of $\eta$ to $X\times\mathbf P^n_k$.

\begin{remark}
    If the Hodge-to-de Rham spectral sequence degenerates at $E_1$, then we must have that $H^0(X,Z^1_{X/k})= H^0(X,\Omega^1_{X/k})$.
    Therefore, by Deligne-Illusie \cite[Theorem 2.1 and Corollary 2.5]{DI}, all examples above do not lift to the ring of second Witt vectors $W_2(k)$.
\end{remark}

\;\\

\subsection*{Acknowledgements}\;

The author thanks Sasha Petrov for very helpful feedback and for pointing out Mumford's example.
This work is supported by an AMS-Simons travel grant.

\;\\

\;\\

\footnotesize{
 \textsc{Department of Mathematics, Yale University,  New Haven, CT, 06511,
USA}\par\nopagebreak
  \textit{E-mail address}: \texttt{siqing.zhang.math@gmail.com}}

\end{document}